\documentclass[12pt,reqno]{amsart}
\usepackage{amssymb,amsfonts,amsbsy}
\usepackage{eucal}
\usepackage{stmaryrd,bbm}
\usepackage{pstricks,pst-node}
\usepackage[width=.90\textwidth,font=small,labelfont=sc]{caption}
\usepackage[bookmarks=false,colorlinks,linkcolor=mylinkcolor,citecolor=mycitecolor]{hyperref} 
	\definecolor{mycitecolor}{rgb}{.1,.5,.1}
	\definecolor{mylinkcolor}{rgb}{0,0,.6}	

	\definecolor{myblue}{rgb}{0,0,.6}
	\definecolor{myred}{rgb}{.6,0,0}
	\definecolor{mycyan}{cmyk}{1,0,0,0}
	\definecolor{mymagenta}{cmyk}{0,1,0,0}

\newtheorem{theo}{Theorem}
\newtheorem{prop}[theo]{Proposition}
\newtheorem{coro}[theo]{Corollary}

\def\mrg{{\magenta|}}

\def\Sym{{\mathfrak S}}
\def\actv{\mathrel{\hbox{$\pmb\cdot_{\!\!\!{\scriptscriptstyle v}}$}}}

\def\K{{\mathbb K}} 
\def\Q{{\mathbb Q}}
\def\N{{\mathbb N}}

\def\id{\mathrm{id}}

\def\charac{\raise 2pt\hbox{$\chi$}}

\def\ncsym{{\mathcal N}}
\def\sym{S^{\Sym}}

\def\cosym{{\mathcal C}}

\def\ab{\hbox{\rm \textbf{ab}}}

\def\shape#1{|#1|}
\def\shape#1{\sort{#1}}
\newcommand{\shift}[2]{{#1^{\scriptstyle+#2}}}
\def\std#1{{#1}^{\scriptstyle\downarrow}}

\def\word{\hbox{\textsl{\small\textsf{w}}}}

\def\atms{\dot{\Pi}}

\def\iDelta{\Delta_{+}}

\def\bs{\boldsymbol}
\def\bb{\raisebox{.10ex}{\small$\bullet$}}

\def\m{{\mathbf m}}
\newcommand{\Hilb}[2]{\mathrm{Hilb}_{#1}(#2)}
\def\x{{\mathbf x}}
\def\y{{\mathbf y}}
\def\z{{\mathbf z}}
\def\q{{\boldsymbol q}}
\def\a{{\boldsymbol a}}
\def\A{{\mathbf A}}
\def\B{{\mathbf B}}
\def\C{{\mathbf C}}
\def\sort#1{#1^{\hskip-1pt \ssearrow}}
\def\sort#1{\smash{\stackrel{\smash{\curvearrowright}}{\smash{#1}}}}
\def\sort#1{\rightthreetimes(#1)}
\def\sort#1{{\lambda}(#1)}
\def\hits{\mathop{{\scriptstyle\amalg}}}
\def\hits{\diamond}

\newcommand{\nshuf}{\mathrel{\raise1pt\hbox{$\scriptscriptstyle\cup{\mskip-4mu}\cup$}}}
\newcommand{\snshuf}{\mathrel{\raise.5pt\hbox{$\scriptscriptstyle\cup{\mskip-4mu}\cup$}}}
\newcommand{\cp}[1]{\mathrel{\#_{#1}\!}}

\newcommand{\djcup}{\ensuremath{\mathaccent\cdot\cup}}

\newpsobject{showgrid}{psgrid}{%
  subgriddiv=1,griddots=10,gridlabels=6pt}
\newif\ifhrule\hrulefalse

\title[Noncommutative invariants and coinvariant space]{%
Invariant and coinvariant spaces for the algebra of symmetric polynomials in non-commuting variables
}

\author{Fran\c{c}ois Bergeron}
\address[Fran\c{c}ois Bergeron]{%
	LaCIM\\ 
	Universit\'e du Qu\'ebec \`a Montr\'eal\\ 
	Case Postale 8888, succursale Centre-ville\\ 
	Montr\'eal (Qu\'ebec) H3C 3P8\\ 
	CANADA 
}
\email{bergeron.francois@uqam.ca}
\thanks{F. Bergeron is supported by NSERC-Canada and FQRNT-Qu\'ebec.}

\author{Aaron Lauve}
\address[Aaron Lauve]{%
	Department of Mathematics, 
	Texas A\&M University,
	College Station, TX\, 77843, 
	USA
}
\email{lauve@math.tamu.edu}

\begin{document}

\begin{abstract}
We analyze the structure of the algebra $\mathbb{K}\langle\mathbf{x}\rangle^{\mathfrak{S}_n}$ of symmetric polynomials in non-commuting variables in so far as it relates to $\mathbb{K}[\mathbf{x}]^{\mathfrak{S}_n}$, its commutative counterpart. Using the ``place-action'' of the symmetric group, we are able to realize the latter as the invariant polynomials inside the former. We discover a tensor product decomposition of $\mathbb{K}\langle\mathbf{x}\rangle^{\mathfrak{S}_n}$ analogous to the classical theorems of Chevalley, Shephard-Todd on finite reflection groups. 
\vskip 2ex

\noindent{\sc R\'esum\'e.}\ Nous analysons la structure de l'alg\`ebre $\mathbb{K}\langle\mathbf{x}\rangle^{\mathfrak{S}_n}$ des polyn\^omes sym\'e\-triques en des variables non-commu\-tatives pour obtenir des analogues des r\'esultats classiques concernant  la structure de l'anneau $\mathbb{K}[\mathbf{x}]^{\mathfrak{S}_n}$ des polyn\^omes sym\'etriques en des variables commutatives. Plus pr\'ecis\'ement, au moyen de ``l'action par positions'', on r\'ealise  $\mathbb{K}[\mathbf{x}]^{\mathfrak{S}_n}$ comme sous-module de $\mathbb{K}\langle\mathbf{x}\rangle^{\mathfrak{S}_n}$. On d\'ecouvre alors une nouvelle d\'ecomposition de $\mathbb{K}\langle\mathbf{x}\rangle^{\mathfrak{S}_n}$ comme produit tensorial, obtenant ainsi un analogues des th\'eor\`emes classiques de Chevalley et Shephard-Todd. 
\end{abstract}

\date{2 October 2009}


\maketitle


\section{Introduction}\label{sec:intro}
One of the more striking results of invariant theory is certainly the following: if $W$ is a finite group of $n\times n$ matrices (over some field $\K$ containing $\Q$), then there is a $W$-module decomposition of the polynomial ring $S=\K[{\x}]$, in  variables ${\x}=\{x_1,  x_2, \ldots, x_n\}$, as a tensor product 
\begin{equation}\label{canonical}
	S \simeq S_W \otimes S^W\,
\end{equation}
if and only if $W$ is a group generated by (pseudo) reflections. 
As usual, $S$ is afforded a natural $W$-module structure by considering it as the symmetric space on the {defining vector space} $X^*$ for $W$, e.g., $w\cdot f(\x) = f(\x \cdot w)$. It is customary to denote by $S^W$ the ring of $W$-invariant polynomials for this action. To finish parsing (\ref{canonical}), recall that $S_W$ stands for the \textbf{coinvariant space}, i.e., the $W$-module 
\begin{equation}\label{eq:coinvariants}
	S_W := S /\big\langle S^W_{+} \big\rangle
\end{equation}
defined as the quotient of $S$ by the ideal generated by constant-term free $W$-invariant polynomials. 
We give $S$ an $\N$-grading by degree in the variables $\x$. Since the $W$-action on $S$ preserves degrees, both $S^W$ and $S_W$ inherit a grading from the one on $S$, and \eqref{canonical} is an isomorphism of graded $W$-modules. One of the motivations behind the quotient in (\ref{eq:coinvariants}) is to eliminate trivially redundant copies of irreducible $W$-modules inside $S$. Indeed, if $\mathcal V$ is such a module and $f$ is any $W$-invariant polynomial with no constant term, then $\mathcal{V}f$ is an isomorphic copy of $\mathcal V$ living within $\big\langle S^W_{+}\big\rangle$. Thus, the coinvariant space $S_W$ is the more interesting part of the story. 

The context for the present paper is the algebra $T=\K\langle\x\rangle$ of noncommutative polynomials, with $W$-module structure on $T$ obtained by considering it as the tensor space on the defining space $X^*$ for $W$. 
In the special case when $W$ is the symmetric group $\Sym_n$, we elucidate a relationship between the space  $S^W$ and the subalgebra $T^W$  of $W$-invariants in $T$.  The subalgebra  $T^W$ was first studied in \cite{BerCoh:1969,Wol:1936} with the aim of obtaining noncommutative analogs of classical results concerning symmetric function theory. Recent work in \cite{BRRZ:2008,RosSag:2006} has extended a large part of the story surrounding (\ref{canonical}) to this noncommutative context. In particular, there is an explicit $\Sym_n$-module decomposition of the form $T \simeq T_{\Sym_n}\otimes T^{\Sym_n}$ \cite[Theorem 8.7]{BRRZ:2008}. See \cite{For:1985} for a survey of other results in noncommutative invariant theory.

By contrast, our work proceeds in a somewhat complementary direction. 
We consider $\ncsym=T^{\Sym_n}$ as a tower of $\Sym_d$-modules under the ``place-action'' and realize $S^{\Sym_n}$ inside $\ncsym$ as a subspace $\Lambda$ of invariants for this action. This leads to a decomposition of $\ncsym$ analogous to (\ref{canonical}). More explicitly, our main result is as follows.

\begin{theo}\label{thm:main} There is an explicitly constructed subspace  $\cosym$ of $\ncsym$ so that $\cosym$ and the place-action invariants $\Lambda$ exhibit a  graded vector space isomorphism
\begin{equation}\label{eq:main}
    	\ncsym \simeq \cosym \otimes \Lambda.
\end{equation}
\end{theo}
An analogous result holds in the case $|\x|=\infty$. An immediate corollary in either case is the Hilbert series formula
\begin{equation}\label{eq:Hilb-t-cosym}
    \Hilb{t}{\cosym}= \Hilb t\ncsym \, \prod_{i=1}^{|\x|} (1-t^i).
\end{equation}
Here, the \textbf{Hilbert series} of a graded space $\mathcal{V}=\bigoplus_{d\geq 0} \mathcal{V}_d$ is the formal power series defined as
    $$\Hilb{t}{\mathcal{V}}=\sum_{d\geq 0} \dim\mathcal{V}_d\, t^d,$$
 where $\mathcal{V}_d$ is the \textbf{homogeneous degree $d$ component} of $\mathcal{V}$. The fact that (\ref{eq:Hilb-t-cosym}) expands as a series in $\N\llbracket t\rrbracket$ is not at all obvious, as one may check that the Hilbert series of $ \ncsym$ is
\begin{equation}\label{eq:bell_ogf}
	\Hilb{t}{\ncsym}=
         1+ \sum_{k=1}^{|\x|} \frac{t^k} {(1-t)(1-2\,t)\cdots (1-k\,t)}\,.
\end{equation}

In Sections \ref{sec:sym} and \ref{sec:ncsym}, we recall the relevant structural features of $S$ and $T$. Section \ref{sec:place action} describes the place-action structure of $T$ and the original motivation for our work. Our main results are proven in Sections \ref{sec:n=infty} and \ref{sec:n<infty}. 
We underline that the harder part of our work lies in working out the case $|\x|< \infty$. This is accomplished in Section \ref{sec:n<infty}. If we restrict ourselves to the case $|\x|= \infty$, both $\ncsym$ and $\Lambda$ become Hopf algebras and our results are then consequences of a general theorem of Blattner, Cohen and Montgomery. As we will see in Section \ref{sec:n=infty}, stronger results hold in this simpler context. For example, (\ref{eq:Hilb-t-cosym}) may be refined to a statement about ``shape'' enumeration. 

\section{The algebra $\sym$ of symmetric polynomials}\label{sec:sym}

\subsection{Vector space structure of $\sym$}
We specialize our introductory discussion to the group $W=\Sym_n$ of permutation matrices (writing $|\x|=n$). The action on $S=\K[\x]$ is simply the \textbf{permutation action} $\sigma\cdot x_i=x_{\sigma(i)}$ and $S^{\Sym_n}$ comprises the familiar symmetric polynomials. We suppress $n$ in the notation and denote the subring of symmetric polynomials by $\sym$. (Note that upon sending $n$ to $\infty$, the elements of $\sym$ become formal series in $\K\llbracket \x\rrbracket$ of bounded degree; we still call them polynomials to affect a uniform discussion.) A monomial in $S$ of degree $d$ may be written as follows: given an $r$-subset $\y = \{y_1,y_2,\ldots ,y_r\}$ of $\x$ and a \textbf{composition} of $d$ into $r$ parts, ${\a}=(a_1,a_2,\ldots,a_r)$ ($a_i>0$), we write $\y^\a$ for $y_1^{a_1}y_2^{a_2}\cdots y_r^{a_r}$. We assume that the variables $y_i$ are naturally ordered, so that whenever $y_i=x_j$ and $y_{i+1}=x_k$ we have $j<k$. Reordering the entries of a composition ${\a}$ in {decreasing} order results in a partition $\sort{\a}$ called the \textbf{shape} of $\a$. Summing over monomials $\y^\a$ with the same shape leads to the monomial symmetric polynomial
\[
	m_{\mu}=m_{\mu}(\x):= \sum_{\sort{\a}=\mu,\ \y\subseteq \x} \! \y^{{\a}}.
\]
Letting $\mu=(\mu_1,\mu_2,\ldots,\mu_r)$ run over all partitions of $d=|\mu|=\mu_1+\mu_2+cdots+\mu_r$ gives a basis for $\sym_d$. As usual, we set $m_0:=1$ and agree that $m_\mu = 0$ if $\mu$ has too many parts (i.e., $n<r$).

\subsection{Dimension enumeration}
A fundamental result in the invariant theory of $\Sym_n$ is that $\sym$ is generated by a family $\{f_k\}_{1\leq k\leq n}$ of algebraically independent   symmetric polynomials, having respective degrees $\deg f_k=k$. (One may choose $\{m_k\}_{1\leq k \leq n}$ for such a family.) It follows that the {Hilbert series} of $\sym$ is 
\begin{equation}\label{eq:hilbert_inv}
	\Hilb{t}{\sym} = \prod_{i=1}^n \frac{1}{1-t^i}.
\end{equation}
Recalling that the Hilbert series of $S$ is $(1-t)^{-n}$, we see from (\ref{canonical}) and (\ref{eq:hilbert_inv}) that the Hilbert series for the coinvariant space $S_{\Sym}$ is the well-known $t$-analog of $n!$:
   \begin{equation}\label{eq:hilbert_coinv}
              \prod_{i=1}^n \frac{1-t^i}{1-t} = \prod_{i=1}^n (1+t+\cdots +t^{i-1}).
    \end{equation} 
In particular, contrary to the situation in (\ref{eq:Hilb-t-cosym}), the series $\Hilb{t}{S}/\Hilb{t}{\sym}$ in $\Q\llbracket t\rrbracket$ \emph{obviously} 
belongs to $\N\llbracket t\rrbracket$.

\subsection{Algebra and coalgebra structures of $\sym$}
Given partitions $\mu$ and $\nu$, there is an explicit multiplication rule for computing the product $m_\mu\cdot m_\nu$. In lieu of giving the formula, see \cite[\S4.1]{BRRZ:2008}, we simply give an example
\begin{equation}\label{eq:m21*m11}
m_{21}\bs{\cdot} m_{11} = 3\,m_{2111}+2\,m_{221}+2\,m_{311}+m_{32}
\end{equation}
and highlight two features relevant to the coming discussion. 

First, we note that if $n<4$, then the first term disappears. However, if $n$ is sufficiently large then analogs of this term always appear with positive integer coefficients. If $\mu=(\mu_1,\mu_2,\ldots,\mu_r)$ and $\nu=(\nu_1,\nu_2,\ldots,\nu_s)$ with $r\leq s$, then the partition indexing the left-most term in $m_\mu m_\nu$ is denoted by $\mu \cup \nu$ and is given by sorting the list $(\mu_1,\ldots, \mu_r,\nu_1,\ldots,\nu_s)$ in increasing order; the right-most term is indexed by $\mu + \nu := (\mu_1+\nu_1,\ldots,\mu_r +\nu_r,\nu_{r+1},\ldots,\nu_s)$. Taking $\mu=31$ and $\nu=221$, we would have $\mu\cup\nu = 32211$ and $\mu+\nu= 531$. 

Second, we point out that the leftmost term (indexed by $\mu\cup\nu$) is indeed a \emph{leading term} in the following sense. An important partial order on partitions takes
\[
	\lambda \leq \mu \qquad\hbox{iff}\qquad 
	\sum_{i=1}^k \lambda_i \ \leq\  \sum_{i=1}^k \mu_i \ \hbox{ for all }k.
\] 
With this ordering, $\mu\cup\nu$ is the least partition occuring with nonzero coefficient in the product of $m_\mu m_\nu$. That is, $\sym$ is \textbf{shape-filtered}: 
$(\sym)_\lambda \cdot (\sym)_\mu \subseteq \bigoplus_{\nu\geq\lambda\cup\mu} (\sym)_\nu$.

The ring $\sym$ is afforded a coalgebra structure with counit $\varepsilon : \sym \to \K$ and coproduct $\Delta : \sym_d \to \bigoplus_{k=0}^{d} \sym_{k} \otimes \sym_{d-k}$ given, respectively, by 
\[
	\varepsilon(m_{\mu}) = \delta_{\mu,0}
	\quad\hbox{ and }\quad 
	\Delta(m_{\mu}) = \sum_{\theta\cup \nu=\mu} m_{\theta} \otimes m_{{\nu}} \,.
\]
If $|\x|=\infty$, $\Delta$ and $\varepsilon$ are algebra maps, making $\sym$  a graded connected Hopf algebra.

\section{The algebra $\ncsym$ of noncommutative symmetric polynomials}
\label{sec:ncsym}

\subsection{Vector space structure of $\ncsym$}
Suppose now that $\x$ denotes a set of non-commuting variables. The algebra $T=\K\langle\x\rangle$ of noncommutative polynomials is graded by degree. A degree $d$ \textbf{noncommutative monomial} $\z\in T_d$ is simply a length $d$ ``word'':
  $$\z=z_1z_2\cdots z_d,  \quad \textrm{with each} \quad z_i\in\x.$$
In other terms, $\z$ is a function $\z:[d]\rightarrow \x$, with $[d]$ denoting the set $\{1,2,\ldots,d\}$. 
The permutation-action on $\x$ clearly extends to $T$, giving rise to the subspace $\ncsym = T^{\Sym}$ of noncommutative $\Sym$-invariants. With the aim of describing a linear basis for the homogeneous component $\ncsym_d$, we next introduce set partitions of $[d]$ and the type of a monomial $\z:[d]\to \x$.
Let $\A = \{A_1, A_2,\ldots, A_r\}$ be a set of subsets of $[d]$. Say $\A$ is a \textbf{set partition} of $[d]$, written $\A \vdash [d]$, iff $A_1\cup A_2\cup\ldots\cup A_r=[d]$, $A_i\neq\emptyset$ ($\forall i$), and $A_i\cap A_j=\emptyset$ ($\forall i\neq j$). 
The \textbf{type} $\tau(\z)$ of a degree $d$ monomial $\z:[d]\to \x$ is the set partition 
\[
	\tau(\z):=\{\z^{-1}(x) \mid  x\in\mathbf{x}\} \setminus \{\emptyset\}\quad \textrm{of}\quad [d],
\]
whose parts are the non-empty fibers of the function $\z$. For instance, 
$$
	\tau(x_1x_8x_1x_5x_8)=\{\{1,3\},\{2,5\},\{4\}\}.
$$
Note that the type of a monomial is a set partition with at most $n$ parts. 
In what follows, we lighten the heavy notation for set partitions, writing, e.g., the set partition $\{\{1,3\},\{2,5\},\{4\}\}$ as $13.25.4$. We also always order the parts in increasing order of their minimum elements.
The \textbf{shape} $\shape{\A}$ of a set partition $\A=\{A_1,A_2,\ldots,A_r\}$ is the (integer) partition  $\shape{|A_1|, |A_2|,\ldots, |A_r|}$ obtained by sorting the part sizes of $\A$ in increasing order, and its \textbf{length} $\ell(\A)$ is its number of parts ($r$).
Observing that the permutation-action is \emph{type preserving}, we are led to index the \textbf{monomial} linear basis for the space $\ncsym_d$ by set partitions:
\[
m_\A=m_{\A}(\x) := \sum_{\tau(\z)=\A,\ \z\in\x^{[d]}} \z
\]
For example, with $n=2$, we have 
$m_{1}=x_1+x_2$, 
$m_{12}=x_1^2+x_2^2$, $m_{1.2}=x_1x_2+x_2x_1$,
$ m_{123} = {x_1}^3 + {x_2}^3 $, 
$ m_{12.3} = {x_1}^2x_2 + {x_2}^2x_1 $, $ m_{13.2} = {x_1}x_2x_1 + {x_2}x_1x_2 $,
$m_{1.2.3}=0$, and so on. (We set $m_{\bs\emptyset} := 1$, taking $\bs\emptyset$ as the unique set partition of the empty set, and we agree that $m_\A = 0$ if $\A$ is a set partition with more than $n$ parts.)

\subsection{Dimension enumeration and shape grading}
Above, we determined that $\dim \ncsym_d$ is the number of set partitions of $d$ into at most $n$ parts. These are counted by the (length restricted) \textbf{Bell numbers} $B_d^{\,(n)}$. Consequently, (\ref{eq:bell_ogf}) follows from the fact that its right-hand side is the ordinary generating function for length restricted Bell numbers. See \cite[\S2]{Kla:2003}. We next highlight a finer enumeration, where we grade $\ncsym$ by shape rather than degree.

For each partition $\mu$, we may consider the subspace  $\ncsym_{\mu}$ spanned by those $m_\A$ for which $\shape{\A}=\mu$. This results in a direct sum decomposition 
$
	\ncsym_d=\bigoplus_{\mu\vdash d} \ncsym_\mu.
$ 
A simple dimension description for $\ncsym_d$ takes the form of a \textbf{shape Hilbert series} in the following manner. View commuting variables $q_i$ as marking parts of size $i$ and set $\q_\mu:=q_{\mu_1}q_{\mu_2}\cdots q_{\mu_r}$. Then
\begin{equation}\label{eq:shape hilbert}
    \Hilb{\q}{\ncsym_d}= \sum_{\mu\vdash d} \dim \ncsym_\mu\, \q_\mu, = \sum_{\A\vdash [d]} q_{\lambda(\A)}.
 \end{equation}
Here, $\q_\mu$ is a marker for set partitions of shape $\shape{\A}=\mu$ and the sum is over all partitions into at most $n$ parts. Such a shape grading also makes sense for $\sym_d$. Summing over all $d\geq0$ and all $\mu$, we get
\begin{equation}\label{eq:shape hilbert lambda}
    \Hilb{\q}{S^\Sym}= \sum_{\mu}  \q_\mu=\prod_{i\geq 1}^n \frac{1}{1-q_i}.
\end{equation}
Using classical combinatorial arguments, one finds the enumerator polynomials $\Hilb{\q}{\ncsym_d}$ are naturally collected in the \textbf{exponential generating function}
\begin{equation}\label{eq:partitions}
    \sum_{d=0}^\infty  \Hilb{\q}{\ncsym_d}\, \frac{t^d}{d!}
      =\sum_{m=0}^n\frac{1}{m!}\left(\sum_{k=1}^\infty q_k \frac{t^k}{k!}\right)^m.
\end{equation}
See \cite[Chap. 2.3]{BerLabLer:1998}, Example 13(a). For instance, with $n=3$, we have
$$	\Hilb{\q}{\ncsym_6} = q_{{6}} +6\,q_5q_1 +15\,q_4q_2
	+15\,q_4q_1^{\!2} + 10\,q_3^{\!2} +60\,q_3q_2q_1
	+15\,{q_2}^{\!3},
$$
thus $\dim \ncsym_{222} =15$ when $n\geq3$. Evidently, the $\q$-polynomials $\Hilb{\q}{\ncsym_d}$ specialize to the length restricted {Bell numbers} $B_d^{\,(n)}$ when we set all $q_k$ equal to $1$.

In view of (\ref{eq:shape hilbert lambda}), (\ref{eq:partitions}), and Theorem~\ref{thm:main}, we claim the following refinement of (\ref{eq:Hilb-t-cosym}).

\begin{coro}\label{thm:multi-graded quotient}
Sending $n$ to $\infty$, the shape Hilbert series of the space $\cosym$ is given by 
\begin{equation}\label{eq:formule shape}
	\Hilb{\q}{\cosym}=
\sum_{d\geq 0}d! \left. \exp\left(\sum_{k=1}^\infty q_k\,\frac{t^k}{k!}\right)\right|_{t^d} \,\,
\prod_{i\geq 1}\bigl(1-q_i\bigr) ,
\end{equation}
with $(\hbox{--})|_{t^d}$ standing for the operation of taking the coefficient of $t^d$.
\end{coro}
This refinement of \eqref{eq:Hilb-t-cosym} will follow immediately from the isomorphism $\cosym \otimes \Lambda \to \ncsym$ in Section \ref{sec:n=infty}, 
which is shape-preserving in an appropriate sense.
Thus we have the expansion
\begin{eqnarray*}
    \Hilb{\q}{\cosym} &=& 1+2\,q_2q_1 + 
    \left( 3\,q_3q_1+2\,{q_2}^{2} +3\,q_2{q_1}^{2} \right) \\
  &&\,\,\,\quad + \left( 4\,q_4q_1 +9\,q_3q_2 +6\,q_3{q_1}^{2} +10\,{q_2}^{2}q_1 +4\,q_2{q_1}^{3} \right)  + \cdots
\end{eqnarray*}

\subsection{Algebra and coalgebra structures of $\ncsym$}\label{subsec:bialgebra structure}
Since the action of $\Sym$ on $T$ is multiplicative, it is straightforward to see that $\ncsym$ is a subalgebra of $T$. The {\em multiplication rule} in $\ncsym$, expressing a product $m_\A \cdot m_\B$ as a sum of basis vectors $\sum_{\C} m_{\C}$, is easy to describe. Since we make heavy use of the rule later, we develop it carefully here. We begin with an example 
(digits corresponding to $\B=\textcolor{myblue}{\bs1}.\textcolor{myblue}{\bs2}$ appear in bold):
\begin{eqnarray}
	\notag m_{\textcolor{myred}{1}\textcolor{myred}{3}.\textcolor{myred}{2}} \,\bs{\cdot}\, m_{\bs{\textcolor{myblue}{1}}.\bs{\textcolor{myblue}{2}}} &=& m_{\textcolor{myred}{1}\textcolor{myred}{3}.\textcolor{myred}{2}.\bs{\textcolor{myblue}{4}}.\bs{\textcolor{myblue}{5}}} + m_{\textcolor{myred}{1}\textcolor{myred}{3}\bs{\textcolor{myblue}{4}}.\textcolor{myred}{2}.\bs{\textcolor{myblue}{5}}} + m_{\textcolor{myred}{1}\textcolor{myred}{3}\bs{\textcolor{myblue}{5}}.\textcolor{myred}{2}.\bs{\textcolor{myblue}{4}}}  \\
\label{eq:productexample} &&\quad\mbox{} +  m_{\textcolor{myred}{1}\textcolor{myred}{3}.\textcolor{myred}{2}\bs{\textcolor{myblue}{4}}.\bs{\textcolor{myblue}{5}}} + m_{\textcolor{myred}{1}\textcolor{myred}{3}.\textcolor{myred}{2}\bs{\textcolor{myblue}{5}}.\bs{\textcolor{myblue}{4}}} + m_{\textcolor{myred}{1}\textcolor{myred}{3}\bs{\textcolor{myblue}{5}}.\textcolor{myred}{2}\bs{\textcolor{myblue}{4}}} + m_{\textcolor{myred}{1}\textcolor{myred}{3}\bs{\textcolor{myblue}{4}}.\textcolor{myred}{2}\bs{\textcolor{myblue}{5}}}
\end{eqnarray}
Compare this to (\ref{eq:m21*m11}). Notice that the shapes indexing the first and last terms in (\ref{eq:productexample}) are the partitions $\shape{13.2} \cup \shape{1.2}$ and $\shape{13.2}+\shape{1.2}$. As was the case in $\sym$, one of these shapes, namely $\shape{\A}+\shape{\B}$, will always appear in the product, while appearance of the shape $\lambda(\A)\cup\lambda(\B)$ depends on the cardinality of $\x$. 

Let us now describe the multiplication rule. Given any $D\subseteq \N$ and $k\in \N$, we write $\shift{D}{k}$ for the set 
\[
	\shift{D}{k}:=\{a+k \mid a\in D\}.
\]
By extension, for any set partition $\A=\{A_1,A_2,\ldots,A_r\}$ we set 
$
	\shift{\A}{k}:=\{\shift{A_1}{k},\shift{A_2}{k},$ $\ldots,\shift{A_r}{k}\}.
$
Also, we set $\A_{\widehat{i}} := \A \setminus \{A_i\}$. 
Next, if $\mathcal X$ is a collection of set partitions of $D$, and $A$ is a set disjoint from $D$, we extend $\mathcal X$ to partitions of $A\cup D$ by the rule
$$
	A\hits \mathcal X := \bigcup_{\B\in\mathcal X} \{A\} \cup \B .
$$
Finally, given partitions $\A=\{A_1,A_2,\ldots,A_r\}$ of $C$ and $\B=\{B_1,B_2,\ldots,B_s\}$ of $D$ (disjoint from $C$), their \textbf{quasi-shuffles} $\A\nshuf \B$ are the set partitions of $C\cup D$ recursively defined by the rules:
\begin{itemize}\itemsep=2pt
   \item[\bb] $\A\nshuf \bs\emptyset = \bs\emptyset\nshuf \A := \A$,  where $\bs\emptyset$ is the unique set partition of the empty set;
   \item[\bb] $\displaystyle \A\nshuf \B := \bigcup_{i=0}^s\, (A_1\cup {B_i}) \hits 
	\Bigl(\A_{\widehat1}\nshuf {(\B_{\widehat{i}})}\Bigr)$, 
taking $B_0$ to be the empty set.
 \end{itemize}
If $\A\vdash[c]$ and $\B\vdash[d]$, we abuse notation and write $\A \nshuf \B$ for $\A \nshuf \shift{\B}{c}$. As shown in \cite[Prop. 3.2]{BRRZ:2008}, the multiplication rule for $m_{\A}$ and  $m_{\B}$ in $ \ncsym$ is
\begin{equation}\label{eq:multiplication rule}
    m_{\A} \cdot m_{\B} = \sum_{\C \in {\A}\,\snshuf\,{\B}} m_{\C}\,.
 \end{equation}
The subalgebra $\ncsym$, like its commutative analog, is freely generated by certain monomial symmetric polynomials $\{m_\A\}_{\A\in \mathcal A}$, where $\mathcal A$ is some carefully chosen collection of set partitions. This is the main theorem of Wolf \cite{Wol:1936}. We use two such collections later, our choice depending on whether or not $|\x|<\infty$. 

The operation $\shift{(\hbox{--})}{k}$ has a left inverse called the \textbf{standardization} operator and denoted by ``$\std{(\hbox{--})}$''. It maps set partitions $\A$ of any cardinality $d$ subset $D\subseteq \N$ to set partitions of $[d]$, by defining $\std{\A}$ as the pullback of $\A$ along the unique increasing bijection from $[d]$ to $D$. For example, $\std{(18.4)} = 13.2$ and $\std{(18.4.67)}=15.2.34$. The coproduct $\Delta$ and counit $\varepsilon$ on $\ncsym$ are given, respectively, by
\[
	\Delta(m_{\A}) = \sum_{\B\djcup\C=\A} m_{\std{\B}} \otimes m_{\std\C}\
	\qquad \hbox{and} \qquad
	\varepsilon(m_{\A}) = \delta_{\A,\bs\emptyset},
\]
where $\B\djcup\C=\A$ means that $\B$ and $\C$ form complementary subsets of $\A$. 
In the case $|\x|=\infty$, the maps $\Delta$ and $\varepsilon$ are algebra maps, making $\ncsym$ a graded connected Hopf algebra.

\section{The place-action of $\Sym$ on $\ncsym$}
\label{sec:place action}

\subsection{Swapping places in $T_d$ and $\ncsym_d$}
On top of the permutation-action of the symmetric group $\Sym_\x$ on $T$, we also consider the ``{place-action}'' of $\Sym_d$ on the degree $d$ homogeneous component $T_d$.  Observe that  the permutation-action of $\sigma\in\Sym_\x$ on a monomial $\z$ corresponds to the functional composition
\[
	\sigma\circ \z:[d]\stackrel{\z}{\longrightarrow} \x\stackrel{\sigma}{\longrightarrow}\x.
\]
By contrast, the \textbf{place-action}   of  $\rho\in \Sym_d$ on   $\z$ gives the monomial
\[
	\z\circ \rho:[d]\stackrel{\rho}{\longrightarrow}[d]\stackrel{\z}{\longrightarrow} \x ,
\]
composing $\rho$ on the right with $\z$. In the linear extension of this action to all of $T_d$, it is easily seen that $\ncsym_d$ (even each $\ncsym_\mu$) is an invariant subspace of $T_d$. Indeed, for any set partition $\A=\{A_1,A_2,\ldots,A_r\}\vdash[d]$ and any $\rho\in\Sym_d$, one has \begin{equation}\label{eq:sagan_rosas}
   m_{\A}\cdot \rho = m_{\rho^{-1} \cdot \A} 
\end{equation}
(see \cite[\S2]{RosSag:2006}), where as usual 
$
  \rho^{-1}\cdot\A:=\{\rho^{-1}(A_1),\rho^{-1}(A_2),\ldots,\rho^{-1}(A_r)\}.
$ 

\subsection{The place-action structure of $\ncsym$}
Notice that the action in (\ref{eq:sagan_rosas}) is shape-preserving and transitive on set partitions of a given shape (i.e., $\ncsym_\mu$ is an $\Sym_d$-submodule of $\ncsym_d$ for each $\mu\vdash d$). It follows that there is exactly one copy of the trivial $\Sym_d$-module inside $\ncsym_\mu$ for each $\mu\vdash d$, that is, a basis for the place-action invariants in $\ncsym_d$ is indexed by partitions.
We choose as basis the polynomials 
\begin{equation}\label{eq:right_inv}
	\m_\mu:=\frac{1}{(\dim{\ncsym_\mu}){\scriptscriptstyle\,}\mu^{\bs !}}\sum_{\shape{\A}=\mu} m_\A,
\end{equation}
with $\mu^{\bs!} = a_1! a_2! \cdots$ whenever $\mu=1^{a_1}2^{a_2}\cdots$. The rationale for choosing this normalizing coefficient 
will be revealed in (\ref{eq:abel}).
 
To simplify our discussion of the structure of $\ncsym$ in this context, 
we will say that $\Sym$ acts on $\ncsym$ rather than being fastidious about underlying in each situation that individual $\ncsym_d$'s are being acted upon on the right by the corresponding group $\Sym_d$. We denote the set $\ncsym^{\Sym}$ of \textbf{place-invariants} by $\Lambda$ in what follows. To summarize,
\begin{equation}\label{eq:right_invts}
\Lambda = \mathrm{span}\!\left\{\textbf{m}_\mu : \mu\hbox{ a partition of }d,\, d\in\N\right\}.
\end{equation}
The pair $(\ncsym,\Lambda)$ begins to look like the pair $(S,\sym)$ from the introduction. This was the observation that originally motivated our search for Theorem~\ref{thm:main}. 

We next decompose $\ncsym$ into irreducible place-action representations. Although this can be worked out for any value of $n$, the results are more elegant when we send $n$ to infinity. Recall that the \textbf{Frobenius characteristic} of a $\Sym_d$-module $\mathcal{V}$ is a symmetric function 
\[
	\mathrm{Frob}(\mathcal{V}) = 
	\sum_{\mu\vdash d} v_\mu \,s_\mu,
\]
where $s_\mu$ is a Schur function (the character of ``the'' irreducible $\Sym_d$ representation $\mathcal V_\mu$ indexed by $\mu$) and $v_\mu$ is the multiplicity of $\mathcal V_\mu$ in $\mathcal{V}$. To reveal the $\Sym_d$-module structure of $\ncsym_\mu$, we use (\ref{eq:sagan_rosas}) and techniques from the theory of combinatorial species.

\begin{prop}\label{prop:charac_ncsym}
For a partition $\mu=1^{a_1}2^{a_2}\cdots k^{a_k}$, having $a_i$ parts of size $i$, we have
\begin{equation}\label{eq:charac_ncsym}
  \mathrm{Frob}(\ncsym_\mu)=
      h_{a_1}[h_1]\,h_{a_2}[h_2]\cdots h_{a_k}[h_k],
  \end{equation}
with $f[g]$ denoting plethysm of $f$ and $g$, and $h_i$ denoting the $i^{\hbox{\small th}}$ homogeneous symmetric function.
\end{prop}
Recall that the \textbf{plethysm} $f[g]$ of two symmetric functions is obtained by linear and multiplicative extension of the rule $p_k[p_\ell]:=p_{k\,\ell}$, where the $p_k$'s denote the usual {power sum} symmetric functions (see \cite[I.8]{Mac:1995} for notation and details). 
\begin{proof}
First recall that the cycle index series for the combinatorial species $\mathbf{P}$ of set partitions is defined as
\[
	Z_{\mathbf{P}} = \sum_{n\geq 0}
        	\sum_{\mu\vdash n}
		\mathrm{fix}(\sigma_\mu) \frac{p_\mu}{z_\mu} \,,
\]
with $\sigma_\mu$ any permutation having cycle structure given by $\mu$, and $\mathrm{fix}(\sigma_\mu)$ denoting the number of fixed points of the action of $\sigma_\mu$ on partitions.
This coincides with $\mathrm{Frob}(\K\mathbf{P})$, since the values of the character of $\K \mathbf{P}$ are the corresponding number of fixed points. Next, realize $\mathbf{P}$ as the composition $\mathbf{E} \circ \mathbf{E}^*$ \cite[Chap 1.4]{BerLabLer:1998} and enumerate $\mathbf{E}^*$ by giving a $\bs q$-weight to each set, according to its cardinality. 
Standard cycle index series identities (including $Z_{\mathbf{E} \circ \mathbf{E}^*}=Z_{\mathbf{E}}[Z_\mathbf{E} \circ Z_{\mathbf{E}^*}]$) then give
\begin{gather*}
	Z_{\mathbf{P}}(\bs q) =\exp \sum_{k\geq 1} \frac1k\biggl(\exp\Bigl(\sum_{j\geq1} q_j^k\,\frac{p_{jk}}j \Bigr)-1\biggr)   
\end{gather*}
(cf. Example 13(c) of Chapter 2.3 in \cite{BerLabLer:1998}). Finally, collecting the terms of weight $\bs q_\mu$ gives $\mathrm{Frob}(\ncsym_\mu)$. For $\mu=1^{a_1}2^{a_2}\cdots k^{a_k}$, we get
\[
	 Z_{\mathbf{P}}(\bs q)\Bigr|_{\bs q_\mu} = \prod_{i=1}^k
		\Bigl(\sum_{\lambda\vdash a_i} \frac{p_\lambda}{z_\lambda}\Bigr)\bigl[\sum_{\nu\vdash i} \frac{p_\nu}{z_\nu} \bigr],
\]
where $z_\mu$ is Macdonald's constant $1^{a_1}a_1!\,2^{a_2} a_2!\,\cdots k^{a_k} a_k!$ and $(-)[-]$ is plethysm. Standard identities between the $h_k$'s and $p_k$'s finish the proof.
\end{proof}
As an example, we consider $\mu=222=2^3$. Since
$$\displaystyle
	h_2=\frac{p_1^2}{2}+\frac{p_2}{2} \qquad \mathrm{and}\qquad
         h_3=\frac{p_1^3}{6}+\frac{p_1p_2}{2}+\frac{p_3}{3},
$$
a plethysm computation (and a change of basis) gives
\begin{eqnarray*}
    h_3[h_2]&=&
	\frac{p_1^3}{6}\left[\frac{p_1^2}{2}+\frac{p_2}{2}\right]
	+\frac{p_1p_2}{2}\left[\frac{p_1^2}{2}+\frac{p_2}{2}\right]
	+\frac{p_3}{3}\left[\frac{p_1^2}{2}+\frac{p_2}{2}\right]\\[4pt]
   &=&\frac{1}{6}\left(\frac{p_1^2}{2}+\frac{p_2}{2}\right)^3
	+\frac{1}{2}\left(\frac{p_1^2}{2}+\frac{p_2}{2}\right)
	     \left(\frac{p_2^2}{2}+\frac{p_4}{2}\right)
	+\frac{1}{3}\left(\frac{p_3^2}{2}+\frac{p_6}{2}\right)\\[4pt]
    &=&s_6+s_{42}+s_{222}.
\end{eqnarray*}
That is, $\ncsym_{222}$ decomposes into three irreducible components, with the trivial representation $s_6$ being the span of $\m_{222}$ inside $\Lambda$.

\subsection{$\Lambda$ meets $\sym$}
\label{sec:newinvariants}
We begin by explaining the choice of normalizing coefficient in (\ref{eq:right_inv}). Rosas and 
Analyzing the \textbf{abelianization} map $\ab: T \rightarrow S$ (the map making the variables $\x$ commute), Rosas and Sagan \cite[Thm. 2.1]{RosSag:2006} show that $\ab\vert_\ncsym$ satisfies: 
\begin{itemize}\itemsep=2pt
\item[\bb] $\ab(\ncsym) = \sym$, and 
\item[\bb] $\ab(m_{\A})$ is a multiple of $m_{\shape{\A}}$ depending only on $\mu=\shape{\A}$, more precisely 
\begin{equation}\label{eq:abel}
	\ab(\mathbf{m}_\mu)=m_\mu. 
\end{equation} 
\end{itemize}
Formula (\ref{eq:abel}) suggests that a natural right-inverse to $\ab\vert_{\ncsym}$ is given by
\begin{equation}\label{def:iota}
    \iota:\sym\hookrightarrow \ncsym,\quad
       \hbox{with }\quad \iota(m_\mu):= \m_\mu.
 \end{equation}
This fact, combined with the observation that $\iota(\sym) = \Lambda$, affords a quick proof of Theorem \ref{thm:main} when $|\x|=\infty$. We explain this now.

\def\usualaction{Whereas, by definition, the first action of $\Sym$ fixes each $m_\A$ ($\sigma\actv m_\A = m_\A$).}
\section{The coinvariant space of $\ncsym$ (Case: \ $|\x|=\infty$)}\label{sec:n=infty}

\subsection{Quick proof of main result}
When $|\x|=\infty$, the pair of maps $(\ab, \iota)$ have further properties: the former is a Hopf algebra map and the latter is a coalgebra map \cite[Props. 4.3 \& 4.5]{BRRZ:2008}. 
Together with \eqref{eq:abel} and \eqref{def:iota}, these properties make $\iota$ a {\bf coalgebra splitting} of $\ab: \ncsym \to \sym \to 0$. A theorem of Blattner, Cohen, and Montgomery immediately gives our main result in this case.

\begin{theo}[{\cite{BlaCohMon:1986}, Thm. 4.14}]\label{thm:crossedproduct} 
If $H \stackrel{\pi}{\longrightarrow} \overline{H} \rightarrow 0$ is an exact sequence of Hopf algebras that is split as a coalgebra sequence, and
the splitting map $\iota$ satisfies $\iota(\bar{1}) = 1$, then $H$ is isomorphic 
to a \emph{crossed product} $A \cp{} \overline{H}$, where $A$ is the \emph{left Hopf kernel} of $\pi$.
\end{theo}

For the technical definition of {crossed products}, we refer the reader to \cite[\S 4]{BlaCohMon:1986}. We mention only that: (i) the crossed product $A \cp{} \overline{H}$ is a certain algebra structure placed on the tensor product $A \otimes \overline{H}$; and (ii) the \textbf{left Hopf kernel} is the subalgebra
$$
	A := \{h\in H : (\id \otimes \pi)\circ \Delta(h) = h \otimes \overline1 \}.
$$
That said, the coinvariants $\cosym$ we seek evidently form the left Hopf kernel of $\ab$.
Before setting off to describe $\cosym$ more explicitly, we point out that $\cosym$ is graded (since the maps $\Delta$, $\id$, and $\ab$ are graded) and so is the map $\cosym \cp{} \Lambda \stackrel{\simeq}{\longrightarrow} \ncsym$ given in the proof of Theorem \ref{thm:crossedproduct} (which is simply $a\otimes \overline h \mapsto a\cdot\iota(\overline h)$). In particular, Theorem \ref{thm:main} follows immediately from this result.

\subsection{Atomic set partitions.}\label{sec:atomics}
Recall the main result of Wolf \cite{Wol:1936} that $\ncsym$ is freely generated by some collection of polynomials. We announce our first choice for this collection now, following the terminology of \cite{BerZab:1}. 
Let $\Pi$ denote the set of all set partitions (of $[d]$, $\forall\,d\geq0$). We introduce the \textbf{atomic set partitions} $\atms$. A set partition $\A=\{A_1,A_2,\ldots,A_r\}$ of $[d]$ is \emph{atomic} if there {does not} exist a pair $(s,c)$ $\,(1\leq s<r, 1\leq c<d)$ such that $\{A_1,A_2,\ldots,A_s\}$ is a set partition of $[c]$. Conversely, $\A$ is not atomic if there are set partitions $\B$ of $[d']$ and $\C$ of $[d'']$ splitting $\A$ in two: $\A = \B\cup\shift{\C}{d'}$. We write $\A=\B \mrg \C$ in this situation. A \textbf{maximal splitting} $\A = \A' \mrg \A'' \mrg \cdots \mrg \A^{(r)}$ of $\A$ is one where each $\A^{(i)}$ is atomic. For example, the partition $17.235.4.68$ is atomic, while $12.346.57.8$ is not. The maximal splitting of the latter would be $12 \mrg 124.35 \mrg 1$, but we abuse notation and write $12 \mrg 346.57 \mrg 8$ to improve legibility.

It is proven in \cite{BerZab:1} that $\ncsym$ is freely generated by the atomic polynomials. To get a better sense of the structure, we order $\Pi$ by giving $\atms$ a total order ``$\prec$'' and then extending lexicographically. Given two atomic set partitions $\A$ and $\B$, we demand that $\A \prec \B$ if $\A \vdash [c]$ and $\B \vdash [d]$ with $c<d$. In case $\A,\B$ are partitions of the same set $[d]$,  any ordering will do for the current purpose\dots\ one interesting choice is to order $\A$ and $\B$ by ordering lexicographically their associated \emph{restricted growth functions} (see Section \ref{sec:rhymes}): if $\A=\{A_1,A_2,\ldots,A_r\} \vdash [d]$, define $\word(\A)\in \N^d$ by
\[
	\word(\A)=w_1w_2\cdots w_d,\quad\hbox{ with }\quad w_i:=k \iff  i\in A_k. 
\]
For example, $\word(\textcolor{myblue}{13}.\textcolor{myred}{2}) = \textcolor{myblue}{1}\textcolor{myred}{2}\textcolor{myblue}{1}$ and $\word(\textcolor{myblue}{17}.\textcolor{myred}{235}.\textcolor{mycyan}{4}.\textcolor{mymagenta}{68}) = \textcolor{myblue}{1}\textcolor{myred}{2}\textcolor{myred}{2}\textcolor{mycyan}{3}\textcolor{myred}{2}\textcolor{mymagenta}{4}\textcolor{myblue}{1}\textcolor{mymagenta}{4}$. The following chain of set partitions of shape $3221$ should adequately illustrate our total ordering on $\atms$:
\[
1\mrg 23 \mrg 45 \mrg 678 \prec 13.2 \mrg 456 \mrg 78 \prec 13.24 \mrg 578.6 \prec 14.23 \mrg 578.6 \prec 17.235.4.68 \prec 17.236.4.58.
\]
In fact, $1\mrg 23 \mrg 45 \mrg 678$ is the unique minimal element of $\Pi$ of shape $3221$. 

Define the \textbf{leading term} of a sum $\sum_\C \alpha_\C\, m_\C$ to be the monomial $m_{\C_0}$ such that $\C_0$ is lexicographically least among all $\C$ with $\alpha_\C \neq 0$. Combined with (\ref{eq:multiplication rule}), our choice for $\prec$ makes it clear that the leading term of $m_{\A} \cdot m_\B$ is $m_{\A\mrg \B}$. 
That is, multiplication in $\ncsym$ is {shape-filtered.} Since the left Hopf kernel $\cosym$ is a subalgebra, it is shape-filtered as well. 
Finally, the isomorphism $\cosym \cp{} \Lambda \stackrel{\simeq}{\longrightarrow} \ncsym$ constructed in the proof of Theorem \ref{thm:crossedproduct} is also shape-filtered. These facts give Corollary \ref{thm:multi-graded quotient} immediately.

\subsection{Explicit description of the Hopf algebra structure of $\cosym$}\label{subsec: Hopf}
We begin by partitioning $\atms$ into two sets according to length,
\[
\atms_{\flat} := \left\{\, \A\in\atms : \ell(\A)=1 \right\}
\qquad\hbox{and}\qquad
\atms_{\sharp} := \left\{\, \A\in\atms : \ell(\A)>1\right\}.
\]
It is easy to find elements of the left Hopf kernel $\cosym$. For instance, if $\A$ and $\B$ belong to $\atms_\flat$, then the Lie bracket $[m_{\A}, m_{\B}]$ belongs to $\cosym$. Indeed,
\begin{align*}
	\Delta\left([m_{\A}, m_{\B}]\right) &\ =\  \Delta\left(m_{\A\mrg\B} - m_{\B\mrg\A} \right) \\
	&\ =\ {m}_{\A\mrg\B}\otimes 1 \ +\  m_{\A} \otimes m_{\B} \ +\  m_{\B} \otimes m_{\A} \ +\  1\otimes {m}_{\A\mrg\B}  \\
	&\phantom{\ =\ \ \ } - {m}_{\B\mrg\A}\otimes1 \ -\  m_{\B} \otimes m_{\A} \ -\  m_{\A} \otimes m_{\B} \ -\  1\otimes {m}_{\B\mrg\A}  \\	
	&\ =\  \left(m_{\A\mrg\B} - m_{\B\mrg\A} \right)\otimes1 +  1\otimes\left(m_{\A\mrg\B} - m_{\B\mrg\A} \right), 
\end{align*}
and
\[
	(\id\otimes\ab)\circ\Delta\left([m_{\B}, m_{\B}]\right) \ =\  
	[m_{\A}, m_{\B}]\otimes1 +  1\otimes0.
\]
Similarly, the sum of monomials $m_{13.2} - m_{12.3}$ (leading term indexed by $13.2\in\atms_\sharp$) belongs to $\cosym$. These two simple examples essentially exhaust the different ways in which an element can belong to $\cosym$. The following discussion makes this precise.


It is proven in \cite{HivNovThi:1} (and independently, in \cite{BerZab:1}) that $\ncsym$ is freely and co-freely generated by the {atomic monomials} $\bigl\{m_\A | \A\in\atms \bigr\}$. A classical theorem of Milnor and Moore \cite{MilMoo:1965} then guarantees that $\ncsym$ is isomorphic to the universal enveloping algebra $\mathfrak{U}(\mathfrak{L}(\atms))$ of the free Lie algebra $\mathfrak{L}(\atms)$ on the set $\atms$. The latter is generated as an algebra by the special primitive elements $\atms$. In the isomorphism $\mathfrak{U}(\mathfrak{L}(\atms)) \stackrel{\simeq}{\longrightarrow} \ncsym$, one obviously maps $\A\in\atms_\flat$ to $m_\A$, for these monomials are already primitive. The choice of where to send $\A\in\atms_\sharp$ is the subject of the next proposition.

\begin{prop}\label{thm:zero-sum prims}
For each $\A\in\atms_\sharp$, there is a primitive element $\tilde m_\A$ of $\ncsym$,  
$$
\displaystyle \tilde{m}_{\A}= m_{\A} - \sum_{\B\in\Pi} \alpha_{\B} \,m_{\B},
$$
satisfying: (i) if $\B \in\atms$ or $\shape{\B}\neq\shape{\A}$, then $\alpha_{\B} = 0$; and\ \  (ii) $\sum_{\B} \alpha_{\B} = 1$.
\end{prop}

\begin{proof} Suppose $\A \in\atms_{\sharp}$. A primitive $\tilde m_\A$ exists by the Milnor--Moore theorem.
\smallskip

\emph{(i).}\ Since $\ncsym = \bigoplus_{\mu} \ncsym_\mu$ is a coalgebra grading by shape, we may assume $\shape{\B} = \shape{\A}$ for any nonzero coefficients $\alpha_\B$. Our total ordering on $\Pi$ shows by triangularity that the $\B$s may be chosen from $\Pi \setminus \atms$. 
\smallskip

\emph{(ii).}\ Define linear maps $\iDelta^{\,j}: \ncsym_{+} \to \ncsym \otimes \ncsym$ recursively by
\begin{align*}
\iDelta(h) &:= \Delta(h) - h\otimes 1 - 1\otimes h, \\[1ex]
\iDelta^{\,j+1}(h) &:= ({\iDelta}\otimes\id^{\otimes j})\circ\iDelta^{\,j}(h) \quad\hbox{for \ }j>0.
\end{align*}
Assume that (i) is satisfied for $\tilde m_\A$ and that $\ell(\A)=r$. Since ${\iDelta}(\tilde{m}_\A)=0$, we have $\iDelta^{\,r}(m_{\A}) = \iDelta^{\,r}(\sum_{\B} \Theta_\B m_\B)$. Now,
\[
\iDelta^{\,r}(m_\A) = \sum_{\sigma\in\Sym_r} \std{A_{\sigma1}} \otimes \std{A_{\sigma2}} \cdots \otimes \std{A_{\sigma r}}.
\]
Indeed, the same holds for any $\B$ with $\shape{\B}=\shape{\A}$:
\[
\iDelta^{\,r}\Bigl(\sum_\B \Theta_\B m_\B\Bigr) =\Bigl(\sum_\B \Theta_\B\Bigr) \sum_{\sigma\in\Sym_r} \std{A_{\sigma1}} \otimes \std{A_{\sigma2}} \cdots \otimes \std{A_{\sigma r}}.
\]
Conclude that $\sum_\B \Theta_B = 1$.
\end{proof}

We say an element $h\in\ncsym_\mu$ has the \emph{``zero-sum''} property if it satisfies (ii) from the proposition. Put $\tilde m_\A := m_\A$ for $\A\in\atms_\flat$. We next describe the coinvariant space $\cosym$. 

\begin{coro}\label{thm:cosym is a Lie kernel}
Let $\mathfrak{C}$ the the Lie ideal in $\mathfrak L(\atms)$ given by $\mathfrak C= \bigl[\mathfrak L(\atms),\mathfrak L(\atms)\bigr] \oplus \atms_\sharp$. If $\varphi: \mathfrak U(\mathfrak L(\atms)) \to \ncsym$ is the Milnor--Moore isomorphism given by $\varphi(\A) = \tilde m_\A$ for all $\A \in \atms$, then $\cosym$ is the Hopf subalgebra $\varphi(\mathfrak U(\mathfrak{C}))$.
\end{coro}

\begin{proof}
We first show that $\varphi(\mathfrak U(\mathfrak C)) \subseteq \cosym$. Indeed, $\tilde m_\A \in \cosym$ for all $\A\in \atms_\sharp$, since the zero-sum property means $\ab(\tilde m_\A)=0$. Moreover, if $\A\in \atms_\flat$ and $p$ is any shape-homogeneous primitive element in $\ncsym$ with $\shape{p}=\shape{\A}$, then $\bigl[\tilde m_\A, p\bigr]$ is primitive and zero-sum on  shapes appearing in the support. Hence, $\bigl[\tilde m_\A,p\bigr]$ also belongs to $\cosym$. The inclusion follows, since $\varphi(\mathfrak U(\mathfrak C))$ is generated by elements of these two types. 

It remains to show that $\cosym \subseteq \varphi(\mathfrak U(\mathfrak C))$. To begin, note that $\mathfrak L(\atms) / \mathfrak C$ is isomorphic to the abelian Lie algebra generated by $\atms_\flat$. The universal enveloping algebra of this latter object is evidently isomorphic to $\sym$ (send $\A=\{[d]\}$ to $m_d$). The Poincar\'e--Birkhoff--Witt theorem guarantees that the map $\varphi(\mathfrak U(\mathfrak C)) \otimes \sym \to \ncsym$ given by $a\otimes b \mapsto a\cdot \iota(b)$ is onto $\ncsym$. Conclude that $\cosym \subseteq \varphi(\mathfrak U(\mathfrak C))$, as needed.
\end{proof}

Before turning to the case $|\x|<\infty$, we remark that we have left unanswered the question of finding a systematic procedure (e.g., a closed formula in the spirit of M\"obius inversion) that constructs a primitive element $\tilde{m}_\A$ for each $\A\in\atms_\sharp$.

\def\rhymescheme{Quoting Bill Blewett from \cite[A000110]{OEIS}, ``a rhyme scheme is a string of letters (eg, $abba$) such that the leftmost letter is always $a$ and no letter may be greater than one more than the greatest letter to its left. Thus $aac$ is not valid since $c$ is more than one greater than $a$. For example, [$\#\Pi_3=5$] because there are 5 rhyme schemes on 3 letters: $aaa, aab, aba, abb, abc$.''}
\section{The coinvariant space of $\ncsym$ (Case: \ $|\x|<\infty$)}\label{sec:n<infty}

\subsection{Restricted growth functions}\label{sec:rhymes}
We repeat our example of Section~\ref{subsec:bialgebra structure}  in the case $n=3$. The leading term with respect to our previous order would be $m_{\textcolor{myred}{13}.\textcolor{myred}{2}.\textcolor{myblue}{4}.\textcolor{myblue}{5}}$, except that this term does not appear because 
$\textcolor{myred}{13}.\textcolor{myred}{2}.\textcolor{myblue}{4}.\textcolor{myblue}{5}$ 
has more than $n=3$ parts. Fortunately, the map $\word$ from set partitions to words on the alphabet $\N_{>0}$ reveals a more useful leading term:
\begin{gather}\label{eq:rhyme-product}
	m_{\bs{\textcolor{myred}{121}}} \,\bs{\cdot}\, m_{{\textcolor{myblue}{12}}} \ =\  0+
	m_{\bs{\textcolor{myred}{121}}\textcolor{myblue}{13}} + m_{\bs{\textcolor{myred}{121}}\textcolor{myblue}{31}} + m_{\bs{\textcolor{myred}{121}}\textcolor{myblue}{23}} + 
	m_{\bs{\textcolor{myred}{121}}\textcolor{myblue}{32}} + m_{\bs{\textcolor{myred}{121}}\textcolor{myblue}{21}} + m_{\bs{\textcolor{myred}{121}}\textcolor{myblue}{12}}\,.
\end{gather}
Notice that the words appearing on the right in \eqref{eq:rhyme-product} all begin by $121$ and that the concatenation $\underline{121}\,\underline{12}$ is the lexicographically smallest word appearing there. This is generally true and easy to see: if $\word(\A)=u$ and $\word(\B)=v$, then $uv$ is the lexicographically smallest element of $\word(\A\nshuf\B)$. 

The map $\word$ maps set partitions to \textbf{restricted growth functions}, i.e., the words $w=w_1w_2\cdots w_d$ satisfying $w_1=1$ and $w_i \leq 1+ \max\{w_1,w_2,$ $\ldots,w_{i-1}\}$ for all $2\leq i\leq d$. 
We call them restricted growth words here. See \cite{Sag:1,Sim:1994,WacWhi:1991} and \cite{BreSch:1,GesWeiWil:1998} for some of their combinatorial properties and applications. These words are also known as \emph{``rhyme scheme words''} in the literature; see \cite{Rio:1979} and \cite[A000110]{Slo:oeis}. 
Before looking for $\cosym$ within $\ncsym$, we first fix the representatives of $\Lambda$. Consider the partition $\mu=3221$. Of course, $\mathbf{m}_\mu$ is the sum of all set partitions of shape $\mu$, but it will be nice to have a single one in mind when we speak of $\textbf{m}_\mu$. A convenient choice turns out to be $123.45.67.8$: if we use the length plus lexicographic order on $\word(\Pi)$, then it is easy to see that $\word(123.45.67.8) = 11122334$ is the minimal element of $\Pi$ of shape $3221$. We are led to introduce the words 
$$
\word(\mu):=1^{\mu_1}2^{\mu_2}\cdots k^{\mu_k} 
$$
associated to partitions $\mu=(\mu_1,\mu_2,\cdots,\mu_k)$; we call such restricted growth words \textbf{convex words} since $\mu_1 \geq \mu_2\geq \cdots \geq \mu_k$.
%
%

\subsection{Proof of main theorem}\label{sec:main2}
Call a restricted growth word a \textbf{primary word} if $w_i\cdots w_{n-1}w_n$ if not a restricted growth word for any $i>1$. The \textbf{maximal splitting} of a restricted growth word $w$ is the maximal deconcatenation $w=w'\mrg w''\mrg\cdots\mrg w^{(r)}$ of $w$ into primary words $w^{(i)}$. For example, $12314$ is primary while $11232411$ is a string of four primary words $1\mrg12324\mrg1\mrg1$. 

It is easy to see that if $a,b,c$, and $d$ are primary, then $ac$ = $bd$ if and only if $a=b$ and $c=d$. Together with the remarks on $\A\nshuf\B$ following \eqref{eq:rhyme-product}, this implies that if $\{u_1,u_2,\ldots, u_r\}$ and $\{v_1,v_2,\ldots, v_s\}$ are two sets of primary words, then 
\[
m_{u_1}m_{u_2}\cdots m_{u_r} \qquad\hbox{and}\qquad
m_{v_1}m_{v_2}\cdots m_{v_s}
\]
share the same leading term (namely, $m_{u_1\mrg u_2\mrg\cdots\mrg u_r}$) if and only if $r=s$ and $u_i=v_i$ for all $i$. In other words, our algebra
$\ncsym$ is \emph{primary word--filtered} and freely generated by the monomials
$
\{m_{\word(\A)} \mid \word(\A) \hbox{ is primary}\}.
$
This is the collection of monomials originally chosen by Wolf. 

We aim to index $\cosym$ by the restricted growth words that don't end in a convex word. 
Toward that end, we introduce the notion of \textbf{bimodal words}, i.e., 
words with a maximal (but possibly empty) convex prefix, followed by one primary word. The \textbf{bimodal decomposition} of a restricted growth word $w$ is the expression of $w$ as a product $w=w'\mrg w''\mrg\cdots\mrg w^{(r)} \mrg w^{(r+1)}$, where $w',w'',\ldots, w^{(r)}$ are bimodal and $w^{(r+1)}$ is a possibly empty convex word (which we call a \textbf{tail}). For a given word $w$, this decomposition is accomplished by first splitting $w$ into primary words, then recombining, from left to right, consecutive primary words to form bimodal words. 
For instance, the maximal splitting of $112212$ is $1\mrg1222\mrg12$. The first two factors combine to make one bimodal word; the last factor is a convex tail: 
$1122212 \mapsto \hbox{%
\begin{pspicture}(1.68,.50)(0,0.11)%
\psset{unit=1em}%
	\rput[bl](0,0){$1\phantom{\mrg}1222\phantom{\mrg}12$}%
	\pscurve[linewidth=.4pt](0,1)(.4,1.1)(.625,1.28)
	  \psline[linewidth=.4pt](.625,1.28)(.625,1)(2.75,1.28)%
	\pscurve[linewidth=.4pt](3.09,1)(3.65,1.1)(4.05,1.28)%
\end{pspicture}%
}$. 
Similarly, 
$$
1231231411122311 \mapsto 123\mrg12314\mrg1\mrg1\mrg1223\mrg1\mrg1 \mapsto \hbox{%
\begin{pspicture}(3.73,0.45)(0,0.10)%
\psset{unit=1em}%
	\rput[bl](0,0){\small$123\phantom{\mrg}12314\phantom{\mrg}1\phantom{\mrg}1\phantom{\mrg}1223\phantom{\mrg}1\phantom{\mrg}1$}%
	\pscurve[linewidth=.4pt](0.05,1)(1.22,1.1)(1.5,1.28)
	  \psline[linewidth=.4pt](1.5,1.28)(1.5,1)(3.92,1.28)%
	\pscurve[linewidth=.4pt](4.28,1)(6.75,1.1)(7.70,1.3)
	  \psline[linewidth=.4pt](7.70,1.3)(7.70,1)(8.18,1.28)%
	\pscurve[linewidth=.4pt](8.57,1)(8.85,1.1)(9.00,1.28)%
\end{pspicture}%
}.
$$

Suppose now that $u$ and $v$ are restricted growth words and that the bimodal decomposition of $u$ is tail-free. Then by construction, the bimodal decomposition of $uv$ is the concatenation of the respective bimodal decompositions of $u$ and $v$. We are ready to identify $\cosym$ as a subalgebra of $\ncsym$. 

\begin{theo} Let $\cosym$ be the subalgebra of $\ncsym$ generated by $\{m_v \mid v\hbox{ is bimodal}\}$. Then $\cosym$ has a basis indexed by restricted growth words $w$ whose bimodal decompositions are tail-free. Moreover, the map $\varphi: \cosym \otimes \Lambda \rightarrow \ncsym$ given by $m_{w'}m_{w''}\cdots m_{w^{(r)}} \otimes \mathbf{m}_{\mu} \mapsto m_{{w'}\mrg {w''}\mrg \cdots\mrg {w^{(r)}} \mrg \word(\mu) }$\, is a vector space isomorphism.
\end{theo}

\begin{proof}
The advertised map is certainly onto, since $\{m_{w} \mid w \in \word(\Pi) \}$ is a basis for $\ncsym$ and every restricted growth word has a bimodal decomposition ${w'}\mrg {w''}\mrg \cdots\mrg {w^{(r)}} \mrg \word(\mu)$. It remains to show that the map is one-to-one.

Note that the monomials $\{m_v \mid v\hbox{ is bimodal}\}$ are algebraically independent: certainly, the leading term in a product $m_{v_1}m_{v_2}\cdots m_{v_s}$ (with $v_i$ bimodal) is $m_{v_1\mrg v_2 \mrg\cdots\mrg v_s}$; now, since every word has a unique bimodal decomposition, no (nontrivial) linear combination of products of this form can be zero. Finally, the leading term in the simple tensor $m_{w'}m_{w''}\cdots m_{w^{(r)}} \otimes \mathbf{m}_{\mu}$ is 
the basis vector $m_{w'\mrg w''\mrg \cdots\mrg w^{(r)}} \otimes {m}_{\word(\mu)}$, so no (nontrivial) linear combination of these will vanish under the map $\varphi$.
\end{proof}


\begin{thebibliography}{10}

\bibitem{BerLabLer:1998}
F.~Bergeron, G.~Labelle, and P.~Leroux.
\newblock {\em Combinatorial species and tree-like structures}, volume~67 of
  {\em Encyclopedia of Mathematics and its Applications}.
\newblock Cambridge University Press, Cambridge, 1998.
\newblock Translated from the 1994 French original by Margaret Readdy, With a
  foreword by Gian-Carlo Rota.

\bibitem{BRRZ:2008}
Nantel Bergeron, Christophe Reutenauer, Mercedes Rosas, and Mike Zabrocki.
\newblock Invariants and coinvariants of the symmetric groups in noncommuting
  variables.
\newblock {\em Canad. J. Math.}, 60(2):266--296, 2008.

\bibitem{BerZab:1}
Nantel Bergeron and Mike Zabrocki.
\newblock The {H}opf algebras of symmetric functions and quasisymmetric
  functions in non-commutative variables are free and cofree.
\newblock {\em J. Algebra Appl.}, 8(4):581--600, 2009.

\bibitem{BerCoh:1969}
G.~M. Bergman and P.~M. Cohn.
\newblock Symmetric elements in free powers of rings.
\newblock {\em J. London Math. Soc. (2)}, 1:525--534, 1969.

\bibitem{BlaCohMon:1986}
Robert~J. Blattner, Miriam Cohen, and Susan Montgomery.
\newblock Crossed products and inner actions of {H}opf algebras.
\newblock {\em Trans. Amer. Math. Soc.}, 298(2):671--711, 1986.

\bibitem{BreSch:1}
David Bremner and Lars Schewe.
\newblock Edge-graph diameter bounds for convex polytopes with few facets.
\newblock preprint, arXiv:0809.0915v2.

\bibitem{For:1985}
Edward Formanek.
\newblock Noncommutative invariant theory.
\newblock In {\em Group actions on rings ({B}runswick, {M}aine, 1984)},
  volume~43 of {\em Contemp. Math.}, pages 87--119. Amer. Math. Soc.,
  Providence, RI, 1985.

\bibitem{GesWeiWil:1998}
Ira Gessel, Jonathan Weinstein, and Herbert~S. Wilf.
\newblock Lattice walks in {${\bf Z}\sp d$} and permutations with no long
  ascending subsequences.
\newblock {\em Electron. J. Combin.}, 5:Research Paper 2, 11 pp.\ (electronic),
  1998.

\bibitem{HivNovThi:1}
Florent Hivert, Jean-{C}hristophe Novelli, and Jean-{Y}ves Thibon.
\newblock Commutative {H}opf algebras of permutations and trees.
\newblock preprint, arXiv: math.CO/0502456.

\bibitem{Kla:2003}
Martin Klazar.
\newblock Bell numbers, their relatives, and algebraic differential equations.
\newblock {\em J. Combin. Theory Ser. A}, 102(1):63--87, 2003.

\bibitem{Mac:1995}
I.~G. Macdonald.
\newblock {\em Symmetric functions and {H}all polynomials}.
\newblock Oxford Mathematical Monographs. The Clarendon Press Oxford University
  Press, New York, second edition, 1995.
\newblock With contributions by A. Zelevinsky, Oxford Science Publications.

\bibitem{MilMoo:1965}
John~W. Milnor and John~C. Moore.
\newblock On the structure of {H}opf algebras.
\newblock {\em Ann. of Math. (2)}, 81:211--264, 1965.

\bibitem{Rio:1979}
John Riordan.
\newblock A budget of rhyme scheme counts.
\newblock In {\em Second {I}nternational {C}onference on {C}ombinatorial
  {M}athematics ({N}ew {Y}ork, 1978)}, volume 319 of {\em Ann. New York Acad.
  Sci.}, pages 455--465. New York Acad. Sci., New York, 1979.

\bibitem{RosSag:2006}
Mercedes~H. Rosas and Bruce~E. Sagan.
\newblock Symmetric functions in noncommuting variables.
\newblock {\em Trans. Amer. Math. Soc.}, 358(1):215--232 (electronic), 2006.

\bibitem{Sag:1}
Bruce Sagan.
\newblock Pattern avoidance in set partitions.
\newblock {\it Ars. Combin.}, to appear.

\bibitem{Sim:1994}
Rodica Simion.
\newblock Combinatorial statistics on noncrossing partitions.
\newblock {\em J. Combin. Theory Ser. A}, 66(2):270--301, 1994.

\bibitem{Slo:oeis}
N.~J.~A. Sloane.
\newblock The on-line encyclopedia of integer sequences.
\newblock published electronically at {\small\tt
  www.research.att.com/~njas/sequences/}.

\bibitem{WacWhi:1991}
Michelle Wachs and Dennis White.
\newblock {$p,q$}-{S}tirling numbers and set partition statistics.
\newblock {\em J. Combin. Theory Ser. A}, 56(1):27--46, 1991.

\bibitem{Wol:1936}
Margarete~C. Wolf.
\newblock Symmetric functions of non-commutative elements.
\newblock {\em Duke Math. J.}, 2(4):626--637, 1936.

\end{thebibliography}


\def\cprime{$'$}

\end{document}